\documentclass[12pt]{article}

\usepackage{amsmath}
\usepackage{amsfonts}
\usepackage{amsthm}
\usepackage{mathrsfs}

\def\bibname{}

\newtheorem{theorem}{Theorem}[section]

\newtheorem{corollary}[theorem]{Corollary}
\newtheorem{proposition}[theorem]{Proposition}
\newtheorem{definition}[theorem]{Definition}

\def\R{\mathbb R}
\def\N{\mathbb N}

\def\ds{{\rm d}s}
\def\d{{\rm d}}
\def\qqfa{\quad{\rm for\ all}\quad}

\def\<{\langle}
\def\>{\rangle}
\def\e{{\rm e}}
\def\dev{{\rm dev}}

\def\op{{\rm op}}
\def\dist{{\rm dist}}

\def\per{{\rm per}}

\def\Re{\mathbb R{\rm e}}

\begin{document}

\title
 {Log-Lipschitz continuity of the vector field on the attractor of
certain parabolic equations}

\author{Eleonora Pinto de Moura and James C. Robinson\\
   Mathematics Institute\\ University of Warwick, Coventry\\ CV4 7AL, UK}



\maketitle

\begin{abstract}
\noindent We discuss various issues related to the
finite-dimensionality of the asymptotic dynamics of solutions of
parabolic equations. In particular, we study the regularity of the
vector field on the global attractor associated with these
equations. We show that certain dissipative partial differential
equations possess a linear term that is log-Lipschitz continuous on
the attractor. We then prove that this property implies that the
associated global attractor $\mathcal A$ lies within a small
neighbourhood of a smooth manifold, given as a Lipschitz graph over
a finite number of Fourier modes. Consequently, the global attractor
$\mathcal A$ has zero Lipschitz deviation and, therefore, there are
linear maps $L$ into finite-dimensional spaces, whose inverses
restricted to $L\mathcal A$ are H\"older continuous with an exponent
arbitrarily close to one.

\end{abstract}

\section{Introduction}

The existence of global attractors with finite upper box-counting
dimension for a wide class of dissipative equations (see Babin and
Vishik \cite{BaVi}, Foias and Temam \cite{FT}, Hale \cite{Ha}, Temam
\cite{T}, for example) strongly suggests that it might be possible
to construct a system of ordinary differential equations whose
asymptotic dynamics reproduces the dynamics on the original
attractor. However, because of the complexity of the flow on the
attractor $\mathcal A$ and its irregular structure, the finite
dimensionality of $\mathcal A$ alone is not immediately sufficient
to guarantee the existence of such a system of ordinary differential
equations.

Indeed, the existence of an ordinary differential equation with
analogous asymptotic dynamics has only been proved for dissipative
partial differential equations that possess an inertial manifold,
i.e. a finite-dimensional, positively invariant Lipschitz manifold
that attracts all orbits exponentially (see Constantin and Foias
\cite{CF}, Constantin \textit{et al.}\ \cite{CFNT}, Foias, Manley
and Temam \cite{FMT}, Foias, Sell and Temam \cite{FSTb}, Temam
\cite{T}, for more details). All the methods available in the
literature construct inertial manifolds as graphs of functions from
a finite-dimensional eigenspace associated with the low Fourier
modes into the complementary infinite-dimensional eigenspace
corresponding to the high Fourier modes.

Foias \textit{et al.}\ \cite{FSTa} showed that if a `certain
spectral gap condition' holds for a given system, then it will
possesses an inertial manifold. Unfortunately, this sufficient
condition is quite restrictive, and there are many equations, such
as the 2D Navier-Stokes equations, that do not satisfy it.
Nonetheless, Kukavica \cite{Kuka03} and \cite{Kuka05} showed that
the global attractor of certain dissipative equations, such as the
Burgers equation in one space dimension,
lies in a Lipschitz graph over a finite number of Fourier modes
independently of the theory of inertial manifolds.

Romanov \cite{Rm} discussed the problem of a finite-dimensional
description of the asymptotic behaviour of dissipative equations
more abstractly. He defined the dynamics on the attractor $\mathcal
A$ to be `finite-dimensional' if there exists a bi-Lipschitz map
$\Pi: \mathcal A \to \R^N$, for some $N$, and an ordinary
differential equation with a Lipschitz vector field on $\R^N$ such
that the dynamics on $\mathcal A$ and $\Pi(\mathcal A)$ are
conjugated under $\Pi$. He then showed that this property is
equivalent to the attractor being contained in a finite-dimensional
Lipschitz manifold, given as a graph over a sufficiently large
number of Fourier modes. Hence, his definition and that of an
inertial manifold are much more similar than they first appear. In
Section 3, we investigate other possible ways to define when the
asymptotic dynamics of solutions of parabolic equations are
`finite-dimensional'. We discuss conditions under which an attractor
is a subset of a Lipschitz manifold given as a graph over a
finite-dimensional space; in particular, we give a concise proof of
an important part of Romanov's result.

To illustrate the problem of constructing a finite set of ordinary
differential equations that reproduces the dynamics on the global
attractor, consider a governing equation $\dot{u}=\mathcal G(u)$
defined on a Hilbert space $H$. Suppose there exists a linear map
$L: H \to \R^N$ that is injective on $\mathcal A$.
In order to study the smoothness of the embedded equation
on $X=L\mathcal A$,
\begin{equation}\label{embed eq}
 \dot{x}=h(x)=L\mathcal GL^{-1}(x),\quad x \in X,
\end{equation}
one needs to consider the continuity of the vector field on
$\mathcal A$ and the continuity of the inverse of the embedding $L$
restricted to $X$.

The regularity of the embedding $L$ has been discussed in a variety
of papers (see Ma\~n\'e \cite{M}, Ben-Artzi \textit{et al.}\
\cite{BEFN}, Eden \textit{et al.}\ \cite{EFNT}, Foias and Olson
\cite{FO}, Hunt and Kaloshin \cite{HK}, Olson and Robinson
\cite{OR}, Robinson \cite{Rb09} for more details).
Hunt and Kaloshin \cite{HK}, for example, showed that if $\mathcal
A$ has finite upper box-counting dimension, there exists a linear
map $L: H \to \R^N$ that is injective on $\mathcal A$ and whose
inverse $L^{-1}: X \to \mathcal A$ is H\"older continuous with
exponent $\alpha$, i.e. there exists $C>0$ such that
\begin{equation}\label{HK}
 C\|L(u)-L(v)\|^{\alpha}\ge \|u-v\|, \qqfa u,v \in \mathcal A.
\end{equation}

Introduced by Assouad \cite{Ass}, the Assouad dimension is another
useful notion of dimension in the study of embeddings of
finite-dimensional sets.
The {\em Assouad dimension} $\dim_A(X)$ of $X$ can be
defined as the infimum over all $d$ for which there exists a
constant $K$ such that
$$\mathcal N(r, \rho) \le K (r/\rho)^d \quad {\rm for \ } 0<\rho<r<1,$$
where $\mathcal N(r, \rho)$ is the number of $\rho$-balls required
to cover any $r$-ball in $\mathcal A$ (for proof see Olson
\cite[Theorem 2.3]{Ol}). For a comprehensive treatment of the
Assouad dimension, see Luukkainen \cite{Luuk}.

The strongest existing embedding result, due to Robinson and Olson
\cite[Theorem 5.6]{OR}, guarantees the existence of an embedding
$L:H \to \R^N$ such that $L^{-1}$ is $\gamma$-log-Lipschitz, i.e.\
there exist $\gamma \ge 0$ and $C>0$ such that
$$\|L^{-1}(x)- L^{-1}(y)\| \le C\|x-y\|{\bigg(\log\frac{M}{\|x-y\|}\bigg)^{\gamma}},\qqfa x,y \in X,$$
where $M$ is a constant depending on $X$, if the set $X-X$ of
differences between elements of $X$ has Assouad dimension
$\dim_A(X-X)<s<N$. However, there is no general method to bound the
Assouad dimension of global attractors associated with dissipative
equations.

In this paper, we will focus our discussion on the regularity of the
vector field $\mathcal G$ in \eqref{embed eq}. If one would like a
system of ordinary differential equations with unique solutions that
generates a flow $\{S_t\}$, then the embedded vector field $h$ in
$X$ does not need to be Lipschitz; it is sufficient for $h$ to be
$1$-log-Lipschitz\footnote{It is, then, possible to extend $h: X \to
\R^N$ to $1$-log-Lipschitz function $\mathcal H: \R^N \to \R^N$ (see
McShane \cite{McS} for details)}. Hence, one needs to show that
there exist
\begin{description}
  \item[(i)] an exponent $\eta>0$ such that the vector field on the attractor
             $\mathcal A$ is $\eta$-log-Lipschitz in $H$, and
  \item[(ii)] an exponent $\gamma>0$ such the inverse of linear embedding $L:H \to \R^N$ is
             $\gamma$-log-Lipschitz when restricted to $X$,
\end{description}
for which the inequality $\eta + \gamma \le 1$ holds so that the
solutions are unique.

It is, therefore, reasonable to consider separately the problem of
the regularity of the vector field on the global attractor
associated with certain parabolic equations. If we assume the very
strong condition
that $L$ is a bi-Lipschitz embedding, then we would only need the
vector field to be $1$-log-Lipschitz to guarantee existence and
uniqueness of solutions of the embedded equation. In Section 4, we
show that certain dissipative partial differential equations, such
as the 2D Navier-Stokes equations, possess a linear term that is
1-log-Lipschitz continuous using methods developed by Kukavica
\cite{Kuka07}.

In Section 5, we prove that the 1-log-Lipschitz continuity of the
linear term implies that there exists a family of Lipschitz
manifolds $\mathcal M_N$ such that the distance between the
$N$-dimensional manifold $\mathcal M_N$ and the attractor $\mathcal
A$ is exponentially small in $N$. (It is interesting to note that
this result does not rely explicitly on the fact that the solutions
of these semilinear equations satisfy the geometric `squeezing
property', introduced by Foias and Temam \cite{FT} and on which many
other constructions depend, eg. Foias \textit{et al.}\ \cite{FMT} or
Pinto de Moura and Robinson \cite{PdeMR09a}).

In Section 6, we show that, for certain dissipative equations, one
can make the H\"older exponent in \eqref{HK} as close to one as
required by taking $N$ sufficiently large.

\section{Notation and general setting}

Consider a dissipative parabolic equation written as an evolution
equation of the form
\begin{equation}\label{eq}
    \frac{\d u}{\d t}+Au=F(u)
\end{equation}
in a separable real Hilbert space $H$ with scalar product
$(\cdot,\cdot)$ and norm $\|\cdot\|$. We suppose that $A$ is a
positive self-adjoint linear operator
with compact inverse and dense domain $D_H(A)\subset H$. For each
$\alpha\ge0$, we denote by $D_H(A^{\alpha})$ the domain of
$A^{\alpha}$, i.e.
 $$D_H(A^{\alpha})=\{u: A^{\alpha}u \in H\};$$
these are Hilbert spaces with inner product
$(u,v)_{\alpha}=(A^{\alpha}u, A^{\alpha}v)$ and norm
$\|u\|_{\alpha}=\|A^{\alpha}u\|$. We know that for $\alpha > \beta$,
the embedding $D_H(A^{\alpha})\subset D_H(A^{\beta})$ is dense and
continuous such that
\begin{equation}\label{embed dom}
\|u\|_{\beta}\le \widetilde{C}(\alpha,\beta)\|u\|_{\alpha},
\quad{\rm for \ } u \in D_H(A^{\alpha})
\end{equation}
(see Sell and You \cite{SeYo}, for details). Moreover, we assume
that, for some $0\le \alpha \le 1/2$, the nonlinear term $F$ is
locally Lipschitz from $D_H(A^{\alpha})$ into $H$, for $u, v  \in
D_H(A^{\alpha}),$
\begin{eqnarray}\label{VF}
  \big\|F(u)-F(v)\big\| &\le& K(R) \big\|A^{\alpha}(u-v)\big\|,
\quad {\rm with} \quad \|A^{\alpha}u\|,\|A^{\alpha}v\| \le R,
\end{eqnarray}
where $K$ is a constant depending only on $R$. This abstract setting includes,
among others, the 2D Navier-Stokes equations
and the original Burgers equation with Dirichlet boundary values
(see Eden {\it et al.}\ \cite{EFNT}, Temam \cite{T} for example).

Since $A$ is self-adjoint and its inverse is compact, $H$ has an
orthonormal basis $\{w_j\}_{j\in \N}$ consisting of eigenfunctions
of $A$ such that
$$Aw_j=\lambda_jw_j \qqfa j\in \N$$
with $0< \lambda_1 \le \lambda_2,...$ and $\lambda_j \to \infty$ as
$j \to \infty.$ With $n \in \N$ fixed, define the finite-dimensional
orthogonal projections $P_n$ and their orthogonal complements $Q_n$
by
$$P_nu=\sum_{j=1}^{n}(u,w_j)w_j \quad {\rm and} \quad Q_nu=\sum_{j=n+1}^{\infty}(u,w_j)w_j.$$
Hence, we can write $u=P_nu+Q_n u$, for all $u \in H.$ The
orthogonal projections $P_n$ and $Q_n$ are bounded on the Hilbert
spaces $D_H(A^{\alpha})$, for any $\alpha>0$ (see \eqref{embed
dom}). Notice that $P_nH=P_nD_H(A^{\alpha})\subset D_H(A^{\alpha})$,
since $P_nH$ is a finite-dimensional subspace generated by the
eigenvectors of $A$ corresponding to the first $n$ eigenvalues of
$A$. These spectral projections commute with the operators
$\e^{-At}$ for $t>0$, i.e., $P_n\e^{-At}=\e^{-At}P_n$ and
$Q_n\e^{-At}=\e^{-At}Q_n$.
Moreover, we have the following estimate
\begin{eqnarray*}
\|\e^{-At}Q_nu\|_{\alpha}&\le&\sup_{j \ge n+1}
  \big\{\lambda_j^{\alpha}\e^{-\lambda_jt}\big\}\|Q_nu\|\le
  b_{n,\alpha}(t)\|Q_nu\|,
\end{eqnarray*}
where
\begin{equation*}
  b_{n,\alpha}(t)=\left\{\begin{array}{l}\bigg(\displaystyle {\frac{\e t}{\alpha}}\bigg)^{-\alpha},
  \quad{\rm for\ }\quad 0<t \le \alpha/\lambda_{n+1}
  \\ \\
  \lambda_{n+1}^{\alpha} \e^{-\lambda_{n+1} t}, \quad{\rm for\ }\quad t\ge \alpha/\lambda_{n+1}
  \end{array}\right.
\end{equation*}
Therefore,
\begin{equation}\label{est1}
 \Big\|A^{\alpha}\e^{-At}Q_n\Big\|_{\mathscr L(H,H)}\le b_{n,\alpha}(t).
\end{equation}

Within this general setting, one can prove the local existence and
uniqueness of solutions of \eqref{eq} (see Henry \cite{Hen} for
details).
In particular, it follows from Henry \cite[Lemma 3.3.2]{Hen} that
the solution of the nonlinear equation \eqref{eq} with initial
condition $u(t_0)=u_0, \ t>t_0$
is given by the variation of constants formula
\begin{equation}\label{9a}
  u(t)=\e^{-A(t-t_0)}u_0+\int_{t_0}^t\e^{-A(t-s)}F\big(u(s)\big)\
  \ds,
\end{equation}
for $t>t_0$ and $u(t_0)\in D_H(A^{\alpha}).$

Thus, we can define $\{\Phi_t\}_{t\ge0}$ to be the semigroup in
$D_H(A^{\alpha})$ generated by \eqref{eq} such that, for any initial
condition $u_0 \in D_H(A^{\alpha})$, there exists a unique solution given by
$u(t;u_0)=\Phi_tu_0.$ We assume that this system is dissipative,
i.e.\ that there exists a compact invariant absorbing set. It
follows from standard results that \eqref{eq} possesses a global
attractor $\mathcal A$, the maximal compact invariant set in
$D_H(A^{\alpha})$ that uniformly attracts the orbits of all bounded
sets (see Babin and Vishik \cite{BaVi}, Hale \cite{Ha}, Robinson
\cite{Rbook}, Temam \cite{T}). So, if $u(0)=u_0\in \mathcal A,$ then
there is a unique solution $u(t)=\Phi_tu_0\in \mathcal A$ that is
defined for all $t \in \R$.

\section{Finite-dimensionality of flows}

Inertial manifolds, as discussed in the Introduction, are a
convenient, although indirect, method to obtain a system of ordinary
differential equations that reproduces the asymptotic dynamics on
the global attractor. Foias, Sell and Temam \cite{FSTa} showed that
if a certain `spectral gap condition' holds -- if there exists an
$n$ such that $ \lambda_{n+1}-\lambda_n>k\lambda_{n+1}^{\alpha}, $
where $k$ is a constant depending on $F$ -- then the system
\eqref{eq} possesses an inertial manifold $\mathcal M$.
Unfortunately, this condition is very restrictive and there are many
equations, such as the 2D Navier-Stokes equations, that do not
satisfy it.

Romanov considered in \cite{Rm} a more general definition of what it
means for a system to be asymptotically finite-dimensional. We will
see that this definition implies the existence of a Lipschitz
manifold that contains the attractor, but does not require it to be
exponentially attracting. Romanov defined the dynamics on a global
attractor $\mathcal A$ to be \emph{finite-dimensional} if for some
$N \ge 1$ there exist:
\begin{description}
  \item[(i)] an ordinary differential equation $\dot{x}=\mathcal
        H(x)$ with a Lipschitz vector field $\mathcal H(x)$ in
        $\R^N$,
  \item[(ii)] a corresponding flow $\{S_t\}$ on $\R^N$ and
  \item[(iii)] a bi-Lipschitz embedding $\Pi:\mathcal A \to \R^N$,
        such that $\Pi(\Phi_tu)=S_t\Pi(u)$ for
        any $u\in \mathcal A$ and $t\ge0.$
\end{description}
It follows from this definition that the evolution operators
$\Phi_t$ are injective on $\mathcal A$ for $t>0.$ If we set
$\Phi_{-t}=\Pi^{-1}S_{-t}\Pi$, then we see that in fact $\Phi_t$ is
Lipschitz on $\mathcal A$ even for $t<0$. Hence, we obtain a
Lipschitz flow $\{\Phi_t\}$ defined on $\mathcal A$ for all $t \in
\R$. In particular, there exist $C \ge 1$ and $\mu>0$ such that
\begin{equation}\label{Lip flow}
  \big\|\Phi_tu-\Phi_tv\big\|_{\alpha}\le C \big\|u-v\big\|_{\alpha}\e^{\mu |t|},
\end{equation}
for every $t \in \R$.

Considering the general Banach space case, Romanov \cite{Rm} proved
that the finite-dimensionality of the dynamics on the attractor
$\mathcal A$ is equivalent to five different criteria. In this
paper, however, we are only interested in the consequences of the
finite-dimensionality of the dynamics on $\mathcal A$. Since our
setting is simpler than Romanov's \cite{Rm}, the arguments involved
in the proof become more transparent. Hence, we include here a
concise and self-contained proof that if the attractor $\mathcal A$
has `finite-dimensional dynamics' in Romanov's sense, then it must
lie on a finite-dimensional manifold, defined as the graph of a
Lipschitz function over $P_nH$, for some $n<\infty$.

\begin{theorem}(Romanov \cite{Rm})\label{rm}
  If the dynamics on $\mathcal A$ is finite-dimensional,
  then, for some $n \ge 1$, there exists a finite-dimensional
  projection $P_n$
  such that
  \begin{equation}\label{GrF}
   \big\|u-v\big\|_{\alpha}\le c\big\|P_n(u-v)\big\|_{\alpha} \qqfa u,v \in \mathcal  A,
  \end{equation}
  where $c=c(\mathcal A, P_n).$
\end{theorem}

\begin{proof}
Consider the variation of constants formula \eqref{9a} with $t=0$
and $u(0)=u \in \mathcal A$.  If we apply the projection operator
$Q_n$ to both sides of \eqref{9a}, then
$$Q_nu=Q_n\e^{At_0}u(t_0)+\int_{t_0}^0\e^{As}Q_nF\big(u(s)\big)\ \ds.$$
Now, since the compact set $\mathcal A$ is bounded in
$D_H(A^{\alpha})$ and $u(t)\in \mathcal A$, it follows from
\eqref{est1} that
$\lim_{t_0\to-\infty}\|Q_n\e^{At_0}u(t_0)\|_{\alpha}=0$.
Consequently, letting $t_0$ tend to $-\infty$ we obtain
$$Q_nu=\int_{-\infty}^0\e^{As}Q_nF(\Phi_su)\ \ds,$$
which converges in $D_H(A^{\alpha})$. It follows from \eqref{Lip
flow} that, for $u,v \in \mathcal A$,
\begin{eqnarray*}
  \big\|Q_nu-Q_nv\big\|_{\alpha}
  &\le&\int_{-\infty}^0\Big\|\e^{As}Q_n\big(F(\Phi_su)-F(\Phi_sv)\big)\Big\|_{\alpha}\
  \ds\\
  &\le&K\int_{-\infty}^0\Big\|A^{\alpha}\e^{As}Q_n\Big\|_{op}\big\|\Phi_su-\Phi_sv\big\|_{\alpha}
  \ds\\
  &\le& KC\big\|u-v\big\|_{\alpha}\int_{-\infty}^0\Big\|A^{\alpha}\e^{As}Q_n\Big\|_{op}
  \e^{\mu s} \ds\\
  \end{eqnarray*}
Using estimate \eqref{est1} with $t=-s$, we find that
\begin{equation*}
  \big\|Q_nu-Q_nv\big\|_{\alpha}\le
  KC\big\|u-v\big\|_{\alpha}\int_{-\infty}^0b_{n,\alpha}(s) \e^{\mu s}
  \ds
\end{equation*}
from which the inequality
\begin{equation}\label{est2}
   \big\|Q_nu-Q_nv\big\|_{\alpha}\le
   \vartheta_n\big\|u-v\big\|_{\alpha},
\end{equation}
 where
$$\vartheta_n:=\frac{1}{KC}\Bigg\{\bigg(\frac{\e}{\alpha}\bigg)^{-\alpha}
   \bigg(\frac{\alpha}{\lambda_{n+1}}\bigg)^{1-\alpha}\frac{1}{1-\alpha}
    + \frac{\lambda_{n+1}^{\alpha}}{\lambda_{n+1}+\mu}\ \e^{\frac{-\alpha(\lambda_{n+1}+\mu)s}{\lambda_{n+1}}}\Bigg\},$$
can be obtained by simple algebraic manipulation.

Note that, since $0<\alpha<1$ and $\lambda_{n+1}$ tends to infinity
as $n\to \infty$, one can choose $n$ sufficiently large to ensure
that $\vartheta_n<1$. Since $P_n+Q_n=I$, it follows that
 \begin{eqnarray*}
   \|Q_n(u_0-v_0)\|_{\alpha} &\le& \vartheta_n \|P_n(u_0-v_0)\|_{\alpha} + \vartheta_n
   \|Q_n(u_0-v_0)\|_{\alpha}\\
   &\le& \frac{\vartheta_n}{1-\vartheta_n}\|P_n(u_0-v_0)\|_{\alpha}.
 \end{eqnarray*}
Thus,
\begin{equation*}
  \|u_0-v_0\|_{\alpha} \le \frac{1}{1-\vartheta_n}\|P_n(u_0-v_0)\|_{\alpha}
\end{equation*}
and hence the attractor is contained in the graph of a Lipschitz
function $\phi:P_nH \to Q_nH.$
\end{proof}

Under the assumption that the non-linear term $F$ is in
$C^2(D_H(A^{\alpha}), H)$, Romanov \cite{Rm} showed that the
finite-dimensionality of $\mathcal A$ implies that the vector field
$\mathcal G(u)=-Au+F(u)$ is Lipschitz\footnote{If $F \in
C^2\big(D_H(A^{\alpha}),H\big)$, then it follows from Henry
\cite[Corollary 3.4.6]{Hen} that the map $(u_0,t) \mapsto u(t)$ is
also in $C^2\big(\R^+\times D_H(A^{\alpha}),D_H(A^{\alpha})\big)$.
Hence, the function $(u_0,t) \mapsto \d u(t)/\d t$ is $C^1$ with
respect to $(u_0,t)$. Since $\d u(t)/\d t=\mathcal G\big(u(t)\big)$,
for a fixed time (we choose $t=1$), the map $u_0\mapsto \mathcal
G\big(u(1)\big)$ is also a $C^1$-function and, consequently, a
Lipschitz function. The finite dimensionality of the dynamics on
$\mathcal A$ implies that the map $u_0 \mapsto u(1)$ is bi-Lipschitz
on $\mathcal A$. And, therefore, the map $u(1)\mapsto \mathcal
G\big(u(1)\big)$ is Lipschitz continuous.}. However, it is not clear
how to adapt his argument to prove that $A$ is Lipschitz. Here we
give a simple argument that shows that finite-dimensionality implies
that the operator $A^\beta$ is Lipschitz in $\mathcal A$, provided
that $\alpha+\beta<1$.

\begin{corollary}
  If the dynamics on $\mathcal A$ is finite-dimensional,
  then, for $\beta$ with $\alpha+\beta<1$, $A^{\beta}$ is Lipschitz on $\mathcal
  A$, i.e.
  $$\big\|A^{\beta}(u-v)\big\|_{\alpha}\le M \big\|u-v\big\|_{\alpha},
  \qqfa u,v \in \mathcal A,$$
where $\alpha$ is given by \eqref{VF}.
\end{corollary}

\begin{proof}
It follows from \eqref{Lip flow} that
$$\big\|u(t)-v(t)\big\|_{\alpha}\le C\e^{\mu|t|}\big\|u_0-v_0\big\|_{\alpha}
\qqfa u_0,v_0 \in D_H(A^{\alpha}).$$ Let $\beta$ be such that
$\alpha+\beta<1$. Then,
\begin{align*}
  \big\|A^{\beta}\big(u(t)-v(t)\big)\big\|_{\alpha}&
  \le
  \big\|A^{\beta}\e^{-At}\big\|_{\op}\big\|u_0-v_0\big\|_{\alpha}\\
  &\ + KC \big\|u_0-v_0\big\|_{\alpha} \int_0^t \Big\|A^{\alpha+\beta}
  \e^{-A(t-s)}\Big\|_{\op}\e^{\mu |s|} \  \ds\\
  &\le \overline{M}\big\|u_0-v_0\big\|_{\alpha} \le M\big\|u(t)-v(t)\big\|_{\alpha}\e^{\mu |s|}
\end{align*}
Since $\mathcal A$ is invariant, given $u, v \in \mathcal A$, we
have that $u=S(t)u_0$ and $v=S(t)v_0$, for some $u_0,v_0\in \mathcal
A$. Hence,
$$\big\|A^{\beta}(u-v)\big\|_{\alpha}\le M \big\|u-v\big\|_{\alpha},
  \qqfa u,v \in \mathcal A.$$
\end{proof}

Although Romanov's result establishes an important criterion for the
`ideal' definition of finite-dimensionality of the dynamics on the
attractor, to require $\mathcal A$ to admit a bi-Lipschitz embedding
into some $\R^N$ is very strong and unlikely to be satisfied in
general.
A sensible way to weaken this definition would be to relax the
bi-Lipschitz assumption and assume the embedded vector field
$\mathcal H$ to be just log-Lipschitz. However, the argument used in
the proof of Theorem \ref{rm} would, then, not work. Hence, it is
more reasonable to remove the assumption that the flow is generated
by an ODE
and define the following:
\begin{definition}
The dynamics on a global attractor $\mathcal A$ is
\emph{finite-dimensional} if, for some $N\ge1$, there exist an
embedding $\Pi:\mathcal A \to \R^N$ that is injective on $\mathcal
A$, a flow $\{S_t\}$ in $\R^N$ and a global attractor $X$, such that
the dynamics on $\mathcal A$ and $X$ are conjugate under $\Phi$ via
$\Pi(\Phi_tu)=S_t\Pi(u),$ for any $u\in \mathcal A$ and $t\ge0.$
\end{definition}
\noindent However, even in this weak sense, it is still an open
problem whether the finite-dimensionality of the global attractor
$\mathcal A$ implies that the dynamics on $\mathcal A$ is
finite-dimensional.

\section{Log-Lipschitz continuity of the vector field}

In the last section, we showed that if the dynamics on the attractor
is finite-dimensional, then $A^{\beta}$ is Lipschitz on $\mathcal A$
provided that $\alpha+\beta<1$, where $\alpha$ is given by
\eqref{VF}. It is relatively easy to show that the converse is also
true (see Robinson \cite{RbICMS}).
\begin{proposition}\label{A^beta Lips}
  Suppose that $A^{\beta}$ is Lipschitz continuous on the attractor
  from $D_H(A^{\alpha})$ into itself, i.e.
  $$\|A^{\beta}u-A^{\beta}v\|_{\alpha} \le M\|u-v\|_{\alpha} \qqfa u,v\in \mathcal A$$
  for some $M>0$. Then, the attractor is a subset of a Lipschitz
  manifold given as a graph over $P_NH$ for some $N$.
\end{proposition}
\noindent The proof of this result follows from a similar argument
to the one developed for the proof of Proposition \ref{AIM} below,
so we omit it here.

Now consider the embedded vector field on $X=L \mathcal A$
$$\dot{x}=h(x)=L\mathcal GL^{-1}(x),\ x \in X.$$ As remarked in the
Introduction, we would like the inverse of the embedding $L$ to be
as smooth as possible and to obtain as much regularity as we can for
$\mathcal G$. However, in general, the regularity of $\mathcal G$ is
determined by the regularity of the linear term $A$, which can be
related to the smoothness of functions on the attractor $\mathcal
A$. For example, it follows from the standard interpolation
inequality
\begin{equation}\label{inter ineq}
  \|Au-Av\|\le\|u-v\|^{1-(1/r)}\|A^r(u-v)\|^{1/r}, \quad {\rm for}\ u,v
  \in \mathcal A,
\end{equation}
that, if $\mathcal A$ is bounded in $D_H(A^r)$, $A$ is H\"older
continuous on $\mathcal A$. In this way, the continuity of $F$ on
$\mathcal A$ can be deduced from the regularity of solutions on the
attractor.

As an example of how one can develop this approach, Foias and Temam
\cite{FT} showed that, in the two dimensional case, the solutions of
the Navier-Stokes equations are analytic in time and proved that the
attractor is bounded in $D_H(A^{1/2}\e^{\tau A^{1/2}})$. Now, if $u
\in D_H(A^{1/2}\e^{\tau A^{1/2}})$, then there exists an uniform
constant $M>0$, such that $\|A^{1/2}\e^{\tau A^{1/2}}u\|^2 < M$.
Hence, $\|A^ku\|^2\le M'(4k)!/(2\tau)^{4k}$, where $M'$ is a
constant depending uniquely on $M$. It follows from \eqref{inter
ineq}, that
$$\|A(u-v)\| \le \bigg[\frac{M'(4j)!}{(2\tau)^{4j}}\bigg]^{1/2j}\|u-v\|^{1-1/j}.$$
If we minimise the right-hand side over all possible choices of $j$,
we obtain that $A: \mathcal A \to H$ is $2$-log-Lipschitz (see
\cite{RbICMS} for example).

This result relies only on the smoothness of solutions. But one can
do much better by making use of the underlying equation. Indeed,
Kukavica \cite{Kuka07} used the structure of the differential
equation \eqref{eq} and far less restrictive conditions on $\mathcal
A$ than above to show that $A^{1/2}: \mathcal A \to H$ is
$1/2$-log-Lipschitz. We briefly outline his argument, which was
primarily developed to study the problem of backwards uniqueness for
nonlinear equations with rough coefficients, and then show that it
can be used to prove that $A: \mathcal A \to H$ is
$1$-log-Lipschitz.

In what follows we will consider the same equation as in Section 2
\begin{equation}\label{eq'}
    \frac{\d u}{\d t}+Au=F(u).
\end{equation}
However, here, we will assume that $\alpha=1/2\ $ such that the
nonlinear term $F$ is locally Lipschitz from $D_H(A^{1/2})$ into
$H$, i.e.
\begin{eqnarray}\label{VF'}
  \big\|F(u)-F(v)\big\| &\le& K(R) \big\|A^{1/2}(u-v)\big\|, \qqfa u, v  \in
  D_H(A^{1/2}),
\end{eqnarray}
with $\|A^{1/2}u\|,\|A^{1/2}v\| \le R$, where $K$ is a constant
depending only on $R$. Moreover, we assume that the maximal
invariant set $\mathcal A$ is bounded in $D_H(A^{1/2})$. The
argument that follows is simple -- the key observation is that the
result is sufficiently abstract that one can make a variety of
choices of $H$ (e.g. we will take $H=L^2$ and $H=H^1$).

Let $u(t)$ and $v(t)$ be solutions of \eqref{eq'}. The equation for
the evolution of the difference $w(t):=u(t)-v(t)$ can be expressed
as
\begin{equation}\label{eq w}
\frac{\d w}{\d t} + Aw = f,
\end{equation}
where $f(t):=F\big(u(t)\big)-F\big(v(t)\big)$. Our assumptions imply
that
  \begin{equation}\label{3}
    \frac{1}{2}\frac{\d}{\d t}(Aw,w)=(w_t,Aw)=-(Aw,Aw)+(f,Aw)
  \end{equation}
  and
  \begin{equation}\label{4}
    \frac{1}{2}\frac{\d}{\d
    t}(Aw,Aw)=(w_t,A^2w)=-(Aw,A^2w)+(f,A^2w).
  \end{equation}
Moreover, it follows from \eqref{VF'} that
  \begin{equation}\label{1}
    \|f\|\le \|F(u)-F(v)\|\le K(\|A^{1/2}u\|\|A^{1/2}v\|)\|A^{1/2}w\|\le K_1  \|A^{1/2}w\|
  \end{equation}
and, consequently,
  \begin{equation}\label{2}
  \Re(f,Aw) \ge -K_2 \|w\|\|A^{1/2}w\|
  \end{equation}
for some $K_1, K_2 \ge 0.$

Under these mild regularity assumptions, Kukavica \cite{Kuka07}
proved the backward uniqueness property, i.e. if $w:[T_0,0]\to H$ is
a solution of \eqref{eq w}, then $w(0) = 0$ implies that $w(t) = 0$
for all $t \in[T_0,0]$. His approach consists in establishing upper
bounds for the log-Dirichlet quotient
$$\widetilde{Q}(t)=\frac{(Aw(t),w(t))}{\|w(t)\|^2\Big(\log\frac{M^2}{\|w(t)\|^2}\Big)},$$
where $M$ is a sufficiently large constant. This quantity is a
variation of the standard Dirichlet quotient
$Q(t)=\|A^{1/2}u\|^2/\|u\|^2$ (see \cite{Og}, \cite{BaTa} for
details). Kukavica showed that, for equations of the form of
\eqref{eq w}, the log-Dirichlet quotient is bounded for all $t\ge 0$
and, as an application of this result, stated the following theorem.
\begin{theorem}[(After Kukavica \cite{Kuka07})]\label{H1 log-DQ}
  Under the above assumptions on the equation \eqref{eq'} with
  $F: D_H(A^{1/2})\to H$ and $\mathcal A \subset D_H(A^{1/2})$,
  there exists a constant $C>0$ such that
  \begin{equation*}
    \|A^{1/2}(u-v)\|^2 \le C \|u-v\|^2\log(M^2/\|u-v\|^2), \qqfa u,v  \in  \mathcal
    A,\ u\neq v,
  \end{equation*}
 where $M=4\sup_{u \in \mathcal A}\|u\|.$
\end{theorem}
\noindent We give a quick summary of Kukavica's proof, filling in
some details in the closing part of the argument.

\begin{proof}[of Theorem \ref{H1 log-DQ}]
  Let $$L(\|w\|)=\log \frac{M^2}{\|w\|^2},$$
  where $M$ is any constant such that
  $$M \ge 4\sup_{u_0 \in \mathcal A}\|u_0\|.$$
  Note that $L(\|w(t)\|) \ge 1$ for all $t \in [0,T_0].$
  For $t \in [0,T_0]$, denote $\widetilde{L}(t)=L\big(\|w(t)\|\big).$
  Define the log-Dirichlet quotient as
  \begin{equation*}
    \widetilde{Q}(t)=\frac{Q(t)}{L\big(\|w\|\big)}
    =\frac{\|A^{1/2}w\|^2}{\|w\|^2L\big(\|w\|\big)}
    =\frac{\|A^{1/2}w\|^2}{\|w\|^2\widetilde{L}(t)}
  \end{equation*}
  where $Q(t)=\|A^{1/2}w\|^2/\|w\|^2$.

  Using \eqref{3} and \eqref{4}, Kukavica \cite{Kuka07} showed in
  the proof of his Theorem 2.1 that
\begin{equation}\label{eq q'}
  \widetilde{Q} '(t) + K_3\widetilde{Q}(t)^2\le K_4,
\end{equation}
for $0<K_3<1$ and $K_4\ge 4K_1^4/4(1-K_3)\ge0$. Applying a variant
of Gronwall's inequality\footnote{Lemma 5.1 (p167 in Temam
\cite{T}):
  \textit{Let $y$ be a positive absolutely continuous function on
  $(0, \infty)$, which satisfies
  \begin{equation}
   y'+\gamma y^p\le \delta
  \end{equation}
  with $p>1$, $\gamma>0$, $\delta \ge 0$. Then, for $t \ge 0$
  \begin{equation}
    y(t) \le
    \bigg(\frac{\delta}{\gamma}\bigg)^{1/p}+\big(\gamma(p-1)t\big)^{-1/(p-1)}.
  \end{equation}}
  } proved in Temam \cite[Lemma 5.1]{T} to \eqref{eq q'},
we obtain that there exists $T$ such that
$$\widetilde{Q}(t) \le C(K_3,K_4), \qqfa t \ge T,$$
where $C(K_3,K_4)$ is a constant independent of $\widetilde{Q}(0)$.

Now, consider $u_0, v_0 \in \mathcal A$. Since solutions in the
attractor exist for all time, we know there exists $t \ge T$ such
that $u_0=S(t)u(-t)$ and $v_0=S(t)v(-t)$ with $u_0\neq v_0$. So,
$u(-t) \neq v(-t)$. Moreover, $\widetilde{Q}(-t)< \infty$ implies
that $\widetilde{Q}(0)\le C(K_3,K_4)$. Hence,
$$\sup_{u_0,v_0  \in  \mathcal A, \  u_0\neq v_0}\widetilde{Q}(t) \le C(K_3,K_4).$$
\end{proof}

We now show that this result can be used to show that $A:\mathcal A
\to H$ is $1$-log-Lipschitz. Write $w=u-v$. If \eqref{1} and
\eqref{2} hold with $H=L^2$, then there exits a constant $C_0>0$
such that
\begin{equation}\label{kuka ineq}
  \|A^{1/2}w\|^2_{L^2} \le C_0\|w\|^2_{L^2}\log\big(M_0^2/\|w\|^2_{L^2}\big),
\end{equation}
where $$M_0 \ge 4\sup_{u \in \mathcal A}\|u\|_{L^2}.$$ This is the
result of Kukavica \cite[Theorem 3.1]{Kuka07} for the 2D
Navier-Stokes equation.

Now assume that $\mathcal A$ is bounded in $D_H(A)$. If \eqref{1}
and \eqref{2} hold with with $H=D_{L^2}(A^{1/2})$, then there exits
a constant $C_1>0$ such that
\begin{equation}
  \|Aw\|^2_{L^2} \le C_1\|A^{1/2}w\|^2_{L^2}\log\big(M_1^2/\|A^{1/2}w\|^2_{L^2}\big)
\end{equation}
where $$ M_1 \ge 4\sup_{u \in \mathcal A}\|A^{1/2}u_0\|_{L^2}.$$ So,
\begin{equation*}
  \|Aw\|_{L^2}^2 \le C_0C_1
  \|w\|_{L^2}^2\log\big(M_0^2/\|w\|_{L^2}^2\big)\log\big(M_1^2/\|w\|_{H^1}^2\big).
\end{equation*}
Since $\|w\|_{L^2} \le \|w\|_{H^1}$,
$$\|Aw\|_{L^2}^2 \le C_0C_1
  \|w\|_{L^2}^2\log\big(M_0^2/\|w\|_{L^2}^2\big)\log\big(M_1^2/\|w\|_{L^2}^2\big).$$
One can choose $M_0$ and $M_1$ such that $M_0 \le M_1$. Hence,
\begin{equation}\label{log-Lipschitz A}
\|Aw\|_{L^2} \le C
  \|w\|_{L^2}\log\big(M_1^2/\|w\|_{L^2}^2\big),
\end{equation}
where $C=\sqrt{C_0C_1}$.
\begin{corollary}\label{logLipscont}
  Under the above assumptions on the equation \eqref{eq'},
  if $\mathcal A$ is bounded in $D_H(A)$, then
  there exists a constant $C>0$ such that
  \begin{equation*}
    \|A(u-v)\| \le C \|u-v\|\log(M_1^2/\|u-v\|^2), \qqfa u,v  \in  \mathcal
    A,\ u\neq v,
  \end{equation*}
 where $M_1\ge4\sup_{u \in \mathcal A}\|A^{1/2}u\|.$
\end{corollary}

Unfortunately, this result is not strong enough to prove the
existence of a smooth finite-dimensional invariant manifold that
contains the attractor. Hence, it would be interesting to know
whether, in such a general setting, the $1$-log-Lipschitz
continuity, obtained for the linear term $A$, is sharp or if it can
be improved. Nevertheless, one can use Corollary \ref{logLipscont}
to show that there exists a family of approximating Lipschitz
manifolds $\mathcal M_N$, given as Lipschitz graphs defined over a
$N$-dimensional spaces, such that the global attractor $\mathcal A$
associated with equation \eqref{eq'} lies within an exponentially
small neighbourhood of $\mathcal M_N$ without a making use of the
squeezing property.

\section{Family of Lipschitz manifolds}


Using the inequality \eqref{log-Lipschitz A} obtained in Section 4,
one can show, for a wide class of parabolic equations, the existence
of a family of Lipschitz manifolds $\mathcal M_N$ such that
$$\dist(\mathcal M_N,\mathcal A)\le C\e^{-k\lambda_{N+1}},$$
where $\mathcal M_N$ is an $N$-dimensional manifold and $C$ and $k$
are positive constants. We obtain this result without appealing to
the squeezing property, on which many constructions in the theory of
inertial manifolds rely (see Foias, Manley and Temam \cite{FMT}, for
example).

\begin{proposition}\label{AIM}
  Suppose that, for some $C>0$,
  \begin{equation}\label{cond}
 \|Aw\|_{L^2} \le C \|w\|_{L^2}\log\big(M_1^2/\|w\|_{L^2}^2\big),
\end{equation}
where $w=u-v$ for $u,v \in \mathcal A$. Then, under the above
conditions on equation \eqref{eq w}, for each $n>0$,
  there exists a Lipschitz function $\Phi_n: P_nH \to Q_nH$,
  $$\|\Phi_n(p_1)-\Phi_n(p_2)\|_{L^2}\le \|p_1-p_2\|_{L^2} \qqfa p_1,p_2 \in P_nH,$$
  such that $\mathcal A$ lies within a $2M_1^2\e^{-\lambda_{n+1}/\sqrt{2}C}$-neighbourhood
  of the graph $\Phi_n$,
 $$\mathbf G[\Phi_n]=\{u \in H: u=p+\Phi_n(p), p\in P_nH\}.$$
\end{proposition}
\noindent Note that the method developed in this proof can also be
used to prove Proposition \ref{A^beta Lips}.
\begin{proof}
Let $w=u-v$, for $u,v \in \mathcal A$.  We can split $w=P_nw+Q_nw$,
and observe that
\begin{align*}
  \|Aw\|_{L^2}^2&=
  \|A(P_nw+Q_nw)\|_{L^2}^2=\|A(P_nw)\|_{L^2}^2+\|A(Q_nw)\|_{L^2}^2\\
  &\ge \lambda_{n+1}^2\|Q_nw\|_{L^2}^2.\\
\end{align*}
It follows from \eqref{cond} that
\begin{align*}
  \|Aw\|_{L^2}^2&\le C^2 \|w\|_{L^2}^2\Big(\log\big(M_1^2/\|w\|_{L^2}^2\big)\Big)^2\\
  &\le C^2 \big(\|P_nw\|_{L^2}^2+\|Q_nw\|_{L^2}^2\big)\Big(\log\big(M_1^2/\|Q_nw\|_{L^2}^2\big)\Big)^2.\\
\end{align*}

Since $\log\big(M_1^2/\|Q_nw\|_{L^2}^2\big) > 1$,
\begin{equation*}
  \frac{\lambda_{n+1}^2\|Q_nw\|_{L^2}^2}{\Big(\log\big(M_1^2/\|Q_nw\|_{L^2}^2\big)\Big)^2}
  \le C^2\|P_nw\|_{L^2}^2+C^2\|Q_nw\|_{L^2}^2
\end{equation*}
Consider a subset $X$ of $\mathcal A$ that is maximal for the
relation
\begin{equation}\label{maximal set}
  \|Q_n(u-v)\|_{L^2}\le \|P_n(u-v)\|_{L^2} \qqfa u,v \in X.
\end{equation}
Note that if the $P_n$ components of $u$ and $v$ agree, so that
$P_nu=P_nv$, then $Q_nu=Q_nv.$ Hence, for every $u \in X$, we can
define uniquely $\phi_n(P_nu)=Q_nu$ such that $u=P_nu
+\phi_n(P_nu)$. Moreover, it follows from \eqref{maximal set} that
\begin{equation*}
  \|\phi_n(p_1)-\phi_n(p_2)\|_{L^2}\le \|p_1-p_2\|_{L^2} \qqfa p_1,p_2 \in P_nX.
\end{equation*}
Standard results (see Wells and Williams \cite{WW}, for example)
allow one to extend $\phi_n$ to a function $\Phi_n: P_nH \to Q_nH$,
that satisfies the same Lipschitz bound.

Now, if $u \in \mathcal A$ but $u \notin X,$ it follows that $$
\|Q_n(u-v)\|_{L^2}\ge \|P_n(u-v)\|_{L^2},$$ for some $v \in X.$
Thus, if $w=u-v$, then
$$\frac{\lambda_{n+1}^2\|Q_nw\|_{L^2}^2}{\Big(\log\big(M_1^2/\|Q_nw\|_{L^2}^2\big)\Big)^2}
\le 2C^2 \|Q_nw\|_{L^2}^2.$$ Hence, $$\|Q_nw\|_{L^2}^2 \le
M_1^2\e^{-\lambda_{n+1}/\sqrt{2}C},$$ which implies that
\begin{align*}
 \|w\|_{L^2}^2&= \|P_nw\|_{L^2}^2+\|Q_nw\|_{L^2}^2 \le 2 \|Q_nw\|_{L^2}^2\\
&\le 2 M_1^2\e^{-\lambda_{n+1}/\sqrt{2}C}. \end{align*} Therefore,
\begin{equation}\label{dist AIM}
 \dist(u, \mathbf G[\Phi_n]) \le 2 M_1^2\e^{-\lambda_{n+1}/\sqrt{2}C}.
\end{equation}

\end{proof}
\noindent A similar statement would hold if one used the inequality
\eqref{kuka ineq} obtained by Kukavica in \cite{Kuka07}, involving
$A^{1/2}$, rather than \eqref{cond} that considers $A$. However, one
would obtain a worse exponent in \eqref{dist AIM}, since
$\lambda_{n+1}$ would be replaced by $\lambda_{n+1}^{1/2}\le
\lambda_{n+1}$.

To illustrate this result, we consider the incompressible
Navier-Stokes equations
\begin{eqnarray*}\label{NSE}
  \partial_tu-\nu\triangle u+u\cdot\nabla u+\nabla p=F,\\
  \nabla\cdot u=0,
\end{eqnarray*} with periodic boundary conditions on
$\Omega=[0,2\pi]^2$ and initial condition $u(x,0)=u_0(t)$ . Here
$u(x,t)$ is the velocity vector field, $p(x,t)$ the pressure scalar
function, $\nu$ the kinematic viscosity and $F(x,t)$ represents the
volume forces that are applied to the fluid. We restrict ourselves
to the space-periodic case for simplicity. Let $\mathcal H$ be the
space of all the $C^{\infty}$ periodic divergence-free functions
that have zero average on $\Omega$. Let $H$ be the closure of
$\mathcal H$
with scalar
product $(\cdot,\cdot)_{L^2}$ and norm $\|\cdot\|_{L^2}$, and let
$V$ be similarly the closure of
$\mathcal H$
with scalar
product $(\cdot,\cdot)_{H^1}$ and norm $\|\cdot\|_{H^1}$. Let $A$ be
the Stokes operator defined by $$Au=- \triangle u,$$ for all $u$ in
the domain $D(A)$ of $A$ in $H$. Now consider the Navier-Stokes
equations written in its functional form
\begin{equation}\label{NSEfunct}
   \frac{\d u}{\d t} + \nu Au+B(u,u)=F,
   \end{equation}
using the operator $A$ and the bilinear operator $B$ from $V\times
V$ into $V'$ defined by
$$(B(u,v),w)=b(u,v,w), \qqfa u,v,w \in V.$$
If $F \in H$ is independent of time, then the equation
\eqref{NSEfunct} possesses a global attractor
\begin{equation*}
 \mathcal A=\bigg\{u_0 \in H: S(t)u_0 {\rm \ exists \ for \ all \ } t
\in
   \R, \  \sup_{t \in \R}\|S(t)u_0\|_{L^2_{\per}(\Omega)}< \infty\bigg\},
 \end{equation*}
 where $S(t)u_0$ denotes a solution starting at $u_0$ on its maximal
 interval of existence (cf. Constantin and Foias \cite{CF}). Under
 these assumptions, the difference of solutions $w=u-v$ will satisfy
   \begin{equation*}
     \frac{\d w}{\d t} + \nu Aw = - \big[ B(w,u) + B(v,w)\big].
   \end{equation*}
 So, in this case we use Kukavica's Theorem with
 $f=- \big[ B(w,u) + B(v,w)\big] $.
Note that
\begin{equation*}
  \|f\|_{H^1} \le K_1 \|A^{1/2}w\|_{H^1},
\end{equation*}
and consequently
 \begin{equation*}
  \Re(f,w) \ge -K_2 \|w\|_{H^1}\|A^{1/2}w\|_{H^1}.
 \end{equation*}
Therefore, one can apply Proposition \ref{AIM} to the two
dimensional Navier-Stokes equation with forcing $F \in L^2$ to show
the existence of a family of approximate inertial manifolds of
exponential order.

\section{Lipschitz Deviation and Embedding Theorem}

The existence of a family of approximating Lipschitz manifolds for a
dissipative equation of the form of \eqref{eq'} implies that the
global attractor $\mathcal A$ has zero Lipschitz deviation, a
concept that we will define below. In this section, we obtain an
embedding of an attractor $\mathcal A$ into $\R^N$ that has H\"older
continuous inverse, and whose exponent can be made arbitrarily close
to one by choosing an embedding space of sufficiently high
dimension.

In 1999, Hunt and Kaloshin \cite{HK}  found an explicit upper bound
for the H\"older exponent $\alpha$ in \eqref{HK}, based on the
thickness exponent. Later, Olson and Robinson \cite{OR} introduced a
variation of this quantity that measures how well a compact set $X$
in a Hilbert space $H$ can be approximated by graphs of Lipschitz
functions (with prescribed Lipschitz constant) defined over a
finite-dimensional subspace of $H$.
\begin{definition}(Olson and Robinson \cite{OR})\label{dev} Let $X$
be a compact subset of a real Hilbert space $H$. Let $\delta_m(X,
\epsilon)$ be the smallest dimension of a linear subspace $U\subset
H$ such that
\begin{equation*}
{\rm dist}(X,\mathbf G_U[\varphi])<\epsilon,
\end{equation*}
for some $m$-Lipschitz function $\varphi:U\rightarrow U^\perp$, i.e.
\begin{equation*}
\|\varphi(u)-\varphi(v)\|\le m\|u-v\|\qqfa u,v\in U,
\end{equation*}
where $U^\perp$ is orthogonal complement of $U$ in $H$ and $\mathbf
G_U[\varphi]$ is the graph of $\varphi$ over $U$:
\begin{equation*}
\mathbf G_U[\varphi]=\{\,u+\varphi(u):\ u\in U\,\}.
\end{equation*}
The  {\em $m$-Lipschitz deviation} of $X$, $\dev_m(X)$,  is given by
\begin{equation*}
\dev_m(X)=\limsup_{\epsilon\rightarrow0}\frac{\log\delta_m(X,\epsilon)}{-\log\epsilon}.
\end{equation*}
\end{definition}
\noindent (Since this quantity is bounded and non-increasing in $m$,
the limit as $m$ tends to infinity exists and is equal to the
infimum.
Indeed, Pinto de Moura and Robinson \cite{PdeMR09a} define the {\em
Lipschitz deviation} of $X$, $\dev(X)$, via $
\dev(X)=\lim_{m\rightarrow\infty}\dev_m(X). $)
Just as in \cite{PdeMR09a}, we show that the existence of a family
of approximating Lispchitz manifolds, such as that provided by
Proposition \ref{AIM}, implies that the associated global attractor
have zero Lipschitz deviation.

Let $\epsilon_n=2M_1^2\e^{-\lambda_{n+1}/\sqrt{2}C}$. It follows
from Proposition \ref{AIM} that the global attractor $\mathcal A$ is
contained in an $\epsilon_n$-neighbourhood of a finite-dimensional
Lipschitz manifold $\mathcal M$, defined as a graph of $\Phi: PH \to
QH$ , with
\begin{equation*}
  |\Phi(p_1)-\Phi(p_2)|\le|p_1-p_2| \qqfa p_1,p_2 \in PH.
\end{equation*}
Hence, $\delta_1(\mathcal A,\epsilon_n)=n$ and
\begin{equation*}
  \limsup_{n \to \infty} \frac{\log \delta_1(\mathcal
A,\epsilon_n)}{-\log  \epsilon_n}=\limsup_{n\rightarrow \infty}
\frac {\log n}{\sigma\lambda_{n+1}-\log c_0}=0.
\end{equation*}
Therefore, the global attractor $\mathcal A$ for a dynamical system
generated by a partial differential equation of the form \eqref{1}
has $\dev_1(\mathcal A)=0$. Consequently, for a wide class of
parabolic equations that satisfy Proposition \ref{AIM}, one can
apply the following abstract embedding result due to Olson and
Robinson \cite{OR}.
\begin{theorem}[(Olson and Robinson \cite{OR})]\label{emb}
Let $\mathcal A$ be a compact subset of a real Hilbert space $H$
with box-counting dimension $d$ and zero Lipschitz deviation. Let
$N>2d$ be an integer and let $\theta$ be a real number with
\begin{equation}\label{alpha}
0<\theta<1-\frac{2d}{N}.
\end{equation}
Then for a prevalent set of linear maps $L:H\rightarrow\R^N$ there
exists a $C>0$ such that
\begin{equation*}
C|L(x)-L(y)|^{\theta}\ge\|x-y\|\qqfa x,y\in \mathcal A;
\end{equation*}
in particular these maps are injective on $\mathcal A$.
\end{theorem}
\noindent Note that the Lipschitz deviation is used to bound
explicitly the H\"older exponent of the inverse of a linear map $L$
restricted to the image of $\mathcal A$. Therefore, in our case, we
can obtain embeddings of $\mathcal A$ into $\R^N$ that have a
H\"older continuous inverse whose exponent is arbitrarily close to
one by taking $N$ sufficiently large.

\section{Conclusion}

In this paper, we studied conditions under which the global
attractor $\mathcal A$ is a subset of a Lipschitz manifold given as
a graph over a finite-dimensional eigenspace of the linear term $A$.
Then, we showed that, since the linear term of a wide class of
dissipative partial differential equations is $1$-Log-Lipschitz
continuous, the associated global attractor $\mathcal A$ lies within
a small neighbourhood of a finite-dimensional Lipschitz manifold.
Consequently, we are able to obtain linear embeddings of the
attractor into $\R^N$, whose inverse is H\"older continuous with
exponent arbitrarily close to one by choosing $N$ sufficiently
large.

The existence of a system of ordinary differential equation whose
asymptotic behavior reproduces the dynamics on an arbitrary
finite-dimensional global attractor remains an interesting open
problem. Nevertheless, if we are able to show that there exist
exponents $\eta>0$ and $\gamma>0$ such that the vector field on the
attractor $\mathcal A$ is $\eta$-log-Lipschitz and the inverse of
linear embedding $L:H \to \R^N$ is $\gamma$-log-Lipschitz when
restricted to $L\mathcal A$, then we will obtain an embedded
equation $\dot{x}=h(x)$ with unique solution, provided $\eta +
\gamma \le 1$.

Robinson and Olson \cite{OR} showed that, for any $\gamma>3/2$, we
can choose $N$ large enough to obtain a $\gamma$-log-bi-Lipschitz
embedding into $\R^N$. However, this lower bound for the exponent
$\gamma$ is too big to ensure uniqueness of solutions. It follows,
however, from results obtained in \cite{PdeMR09b} that the exponent
$\gamma$ cannot be made smaller than $1/2$. Therefore,
\begin{description}
  \item[-] we would like to improve the exponent $1$ in Corollary
  \ref{logLipscont} and
  \item[-] we would like to reduce the exponent $\gamma$ of the logarithmic term in
        \cite{OR}.
\end{description}

Finally note that if the result for $A$ is optimal, then we need a
bi-Lipschitz embedding to guarantee uniqueness of solutions.
However, that Romanov \cite{Rm} obtained a better regularity result
for the vector field, than we obtained for the linear term $A$,
suggests that it may be possible to improve the logarithmic exponent
in our result.

\section*{Acknowledgments}

EPM is sponsored by CAPES and would like to thank CAPES
for all their support during her PhD. JCR is supported by an EPSRC
Leadership Fellowship EP/G007470/1.

\end{document}